\documentclass[12pt]{amsart}
\usepackage{amsmath}
\usepackage{appendix}

\date{\today}
\usepackage{float}

\usepackage[left=1in,right=1in,top=1in,bottom=1in]{geometry}
\usepackage[switch*,pagewise]{lineno}
\usepackage{amsthm}
\usepackage{amscd,amssymb}
\usepackage[all]{xy}
\usepackage{color}
\usepackage{appendix} 
\usepackage{verbatim}
\usepackage{hyperref}
\usepackage{pdfsync}
\usepackage{mathrsfs}
\usepackage{tikz, wrapfig,array}
\usetikzlibrary{calc,patterns,
	decorations.pathmorphing,cd,
	decorations.markings, shapes.geometric, graphs, graphs.standard, quotes,shapes,chains,scopes,positioning,arrows}
\usepackage[normalem]{ulem}
\newcommand{\po}{\ar@{}[dr]|{\text{\pigpenfont R}}} \newcommand{\pb}{\ar@{}[dr]|{\text{\pigpenfont J}}}

\usepackage[noabbrev]{cleveref}
 \newtheorem{thm}{Theorem}[section]

 \newtheorem{lemma}[thm]{Lemma}
 \newtheorem{cor}[thm]{Corollary}
 \newtheorem{prop}[thm]{Proposition}
  
\theoremstyle{definition}
 
 \newtheorem{rmk}[thm]{Remark}
 \newtheorem{defn}[thm]{Definition}

\newcommand{\Hom}{{\rm Hom}}
\newcommand{\conn}{{\rm conn}}

\newcommand{\N}{{\rm N}}

\def\cG{\mathcal{G}}

\def\cC{\mathcal{C}}

\newcommand{\ifif}{if and only if }
\def\x{\times}
\def\xhe{$\x$\text{-homotopy equivalence}}

\def\wrt{\text{with respect to }}
\def\llp{\text{left lifting property}}

\def\hep{\text{homotopy extension property}}
\def\hlp{\text{homotopy lifting property}}

\begin{document}
		\title[(Lack of)  Model Structures on $\cG$]{(Lack of)  Model Structures on the Category of Graphs}
	\date{}
	\author{Shuchita Goyal } 
	\address{Chennai Mathematical Institute} \email{ shuckriya.goyal@gmail.com} 
	\author{ Rekha Santhanam } \address{Dept. of Mathematics, Indian Institute of Technology Bombay} \email{ reksan@iitb.ac.in} 
	
	\begin{abstract}
	     In this article, we study model structures on the category of finite graphs with $\x$-homotopy equivalences as the weak equivalences. We show that there does not exist an  analogue of Str\o{}m-Hurewicz model structure on this category of graphs. More interestingly, we show that this category of graphs  with $\x$-homotopy equivalences does not have a model structure whenever the class of cofibrations is a subclass of  graph inclusions. 	

	\end{abstract}
	
		\maketitle

	\section{Introduction}

Let $\cG$ denote the category whose objects are finite undirected graphs without multiple edges and morphisms are edge preserving functions on vertices.  The class of $\x$-homotopy equivalences in $\cG$ defines a class of weak equivalences on  the category of graphs, and we denote this category with weak equivalences as $(\cG, \x)$.  These equivalences were defined by Dochtermann in \cite{x-htpy}  while extending the work of Lov\'{a}sz, Babson, Kozlov  on Hom complexes of graphs.

	Hom complexes are of  interest as their connectivity often  gives a lower bound on the chromatic number of a graph.   
	A graph $T$ is a test graph if the following inequality holds for every graph $G$,
	\begin{equation}\label{lovasz conjecture}
		\chi(G) \ge \chi(T) + \conn(\Hom(T,G)) + 1.
	\end{equation}
	Here,  $\Hom(T,G)$ denotes the Hom complex of  graphs $T$ and $G$; and $\conn(X)$ is the smallest integer $n$, for which there exists a map from the $(n+1)$-sphere $S^{n+1} \to X$ which cannot be extended to a map from the $(n+2)$-disk $D^{n+2} \to X$. In particular, complete graphs, cycle graphs and complete bipartite graphs are all examples of  test graphs.
	
In general, the problem of computing Hom complexes can have high complexity \cite{pspace-complete}. 	Dochtermann \cite{x-htpy} showed that any graph map $G \to H$ is a $\x$-homotopy equivalence \ifif  the induced map on Hom complexes, $\Hom(T,G) \to \Hom(T,H)$,  is a homotopy equivalence for every graph $T$.	Our  main motivation for the present work is to  be  able to replace graphs with $\x$-homotopy equivalent graphs whose Hom complexes  would be  easier to compute.

Our goal is then to be able to write a graph as a pushout of smaller subgraphs 
such that this pushout is mapped to a homotopy pushout in spaces. In  \cite{Self-htpy}, we noted that given any graph $T$, the functor $\Hom(T,\underline{\ \ })$  maps the double mapping cylinder, $ B \sqcup_f( A\times I_m) \sqcup_g C$,  to $\Hom(T,B) \sqcup_{f_T} ( \Hom(T,A) \times I )\sqcup_{g_T} \Hom(T,C)$.
 In topological spaces, when the maps in the pushout diagram are Hurewicz cofibrations \cite[Proposition 6.2.6]{arkowitz}, the pushout is a homotopy pushout, that is, a pushout which is preserved (up to a weak equivalence) under weakly equivalent diagrams. 
However, there is no known criterion in graphs to  recognise when a given graph is $\x$-homotopic to a double mapping cylinder of smaller graphs.
Further, as explained in \cite{Self-htpy} the double mapping cylinder is not  a correct notion of homotopy pushout in ($\cG, \x$). 

A model structure on $(\cG,\x)$ will allow us to define the notion of homotopy pushout and will give a criterion to recognise when a pushout is a homotopy pushout (that is, pushouts where the maps are cofibrations). In order to resolve these questions, we require a model structure on $(\cG,\x) $ with the property that the $\Hom(T,\underline{\ \ })$ functor will map  homotopy pushouts in the category  of graphs to  homotopy pushouts in the category  of spaces.  The category of infinite graphs is known to have   model structures, albeit with different  choices of weak equivalences. 
Droz constructed different  model structures on the category of graphs for which the homotopy type of a graph was the set of its connected components \cite[Theorem 4.2]{quillen}, its furbished part \cite[Theorem 4.4]{quillen}, its corresponding core graph \cite[Theorem 4.13]{quillen}. 
 Matsushita used the usual model structure on $Top$ to construct a model structure on the category of graphs in \cite{matsushita-model-category} and showed that it is Quillen equivalent to the category of $\mathbb{Z}_2$-spaces. The class of weak equivalences considered in \cite{matsushita-model-category} is the class of maps that induce a $\mathbb{Z}_2$-homotopy equivalence on the box complex of graphs (which for a graph $G$, is a  simplicial complex homotopy equivalent to $\Hom(K_2,G)$).  The class of $\x$-homotopy equivalences is a proper subclass of this class of weak equivalences.
 
Since the notion of $\x$-homotopy equivalence resembles homotopy equivalence from {\it Top}, the most natural model structure to expect on $(\cG,\x)$  is the analogue of the Str\o{}m-Hurewicz model structure on topological spaces. 
In \cite{strom}, Str\o{}m showed the existence of a model structure on {\it Top} with closed Hurewicz cofibrations, Hurewicz fibrations and homotopy equivalences as the class of cofibrations, fibrations and weak equivalences, respectively. 
We show that there is  no analogue of the Str\o{}m-Hurewicz model structure  on $(\cG,\x)$ (cf. Theorem \ref{no hurewicz model structure}). 

Based on our motivation to break down a graph as a homotopy pushout of smaller graphs,  we want  the cofibrations in our model structure on $(\cG, \x)$ to be a subclass of inclusions (up to isomorphisms) in $\cG$. In Remark \ref{rem:induced}, we point out that in order to satisfy the required axioms, we need cofibrations to be induced inclusions. Further, we note that (cf. Remark \ref{rem:not all induced inclusions can be cofibs})  the class of cofibrations for any model structure on $(\cG, \x)$  cannot be the class of  all induced inclusions. In Proposition \ref{cobase change of bad folds result}, we show that if cofibrations are any subclass of induced inclusions in graphs, then every acyclic cofibration has to be a composition of unfolds. This result then allows us to prove Theorem \ref{cof as any subclass of induced inclusions then no model structure} which states that there is no model structure on $(\cG, \x)$ even if we restrict our class of cofibrations to a subclass of induced inclusions. The proof is based on the lack of compatibility between cofibrations and fibrations in the given scenario. 

We restrict ourselves to the category of finite graphs, since to reach  $\x$-homotopy equivalent graphs we only need a finite number of folds and unfolds. In the case of infinite graphs, there are other ways to show that $\x$-homotopy equivalences do not give a good class of weak equivalences to form a model structure. In both the scenarios, whether we consider the category of finite or infinite graphs, the main obstruction to having a model structure seems to be that acyclic cofibrations  will be a composition of unfolds. We study whether we can  instead have a cofibration category structure on $(\cG,\x)$ in a forthcoming article.

	\section{Homotopy Extensions in Graphs}

	A {\it graph} $G$ is a pair of sets $(V(G),E(G))$. The elements of $V(G)$ are called {\it vertices} of $G$, and  the elements of $E(G)$, which are  two-element subsets of $V(G)$, are called {\it edges} of $G$. These elements need not be distinct, that is, $\{y,y\}$ can also be an edge and this makes $y$ into a {\it looped vertex}. 
	We  say that two vertices  $x$, $y$ are {\it adjacent} if  $\{x,y\} \in E(G)$ and denote the edge by $xy$.  A  {\it simple graph} is a graph without any looped vertex and multiple edges. A {\it complete graph} is a simple graph where any two distinct vertices are adjacent to each other. A {\it finite graph} is a graph with a finite set of  vertices.
	A {\it graph map} between two graphs $f: G \to H$ is a function from $V(G)$ to $V(H)$ such that if $xy$ is an edge in $G$, then $f(x)f(y)$ is an edge in $H$.  
	
	\begin{defn}
	    Let $G , H \in \cG$ be two graphs. The (categorical) product of $G$ and $H$ is defined to be the graph $G \x H$ whose vertex set is the cartesian product, $V(G) \x V(H)$, and $(g,h)(g',h') \in E(G \x H)$ whenever $gg' \in E(G)$ and $\ hh' \in E(H) $.
	\end{defn}
	Let $G$ be a graph and $v \in V(G)$ be a vertex of $G$. The {\it neighbourhood set} of $v$ in $G$ is the set $\{x \in V(G): vx \in E(G)\}$, and is denoted by $\N_G(v)$. We drop the  subscript $G$ in  $\N_G(v)$, and we write $\N(v)$ whenever the choice of graph is clear.
	\begin{defn}\label{defn-fold,stiff,contractible}
		A vertex $v$ of a graph $ G$ is said to fold to a vertex $v' \in V(G)$ if every neighbour of $v$ is also a neighbour of $v'$, that is, $\N(v) \subseteq \N(v')$. We note that the map $f : G \to G-v$ that maps each vertex (other than $v$) of $G$ to itself, and $v$ to $v'$ is a graph map. In such a case, we call $f$ a fold map that folds  $G$ to $G-v$. If a graph folds to a single looped vertex via a sequence of fold maps, then it is called a contractible graph. Dual to a fold map, an unfold is defined to be the inclusion map $A \to A \cup \{v\}$ such that $\N_{A \cup \{v\}}(v) \subseteq \N_A(v')$ for some $v' \in V(A)$, where $v \notin V(A)$.
	\end{defn}
	
    \begin{defn}
   A graph is stiff if there is no fold in that graph.  By a  stiff subgraph of a graph $G$ we mean a subgraph of $G$ which is a stiff graph and can be obtained from $G$ via a sequence of folds.
     It is easy to observe that if $S \subset G$ is a stiff subgraph of $G$ and $S$ is simple, then $G$ is also a simple graph. 
     \end{defn}
     
	For $n \in \mathbb{N}$, let $I_n$ be the graph with vertex set $\{0,1,\dots,n\}$ and edge set $\{ij : |i-j| \le 1\}$. We note that the graph  $I_n$ folds down to a single looped vertex and hence is contractible for any $n \in \mathbb{N}$.

	\begin{defn} \cite[Definition 4.3]{x-htpy}
		Two graph maps $f,g : A \to B$ are said to be $\x$-homotopic if for some $n \in \mathbb{N}$, there exists a graph map $F : A \x I_n \to B$ such that $F|_{A\x \{0\}} = f$ and $F|_{A\x \{n\}} = g$. Two $\x$-homotopic maps are denoted as $f \simeq_{\x} g$. A graph map $f : A \to B$ is a $\x$-homotopy equivalence, if there exists a graph map $g: B \to A$ such that $gf \simeq_{\x} 1_A$ and $fg \simeq_{\x} 1_B$, where $1_X$ denotes the identity map on $X$.
	\end{defn}
	
 It is easy to see that the fold and unfold maps are $\x$-homotopy equivalences. Thus every graph is $\x$-homotopy equivalent to its stiff subgraph. Moreover, if $f : A \to B$ is a $\x$-homotopy equivalence between stiff graphs $A$ and $B$, then $f$ is an isomorphism \cite[Thm 6.6]{x-htpy}. 
\begin{lemma}\label{stiff copy}
Let $f: G \rightarrow H$ be a $\x$-homotopy equivalence of graphs. Let  $G'$ and $H'$ be stiff subgraphs of $G$ and $H$ respectively. Then there exists a commuting diagram  as shown below, where $f' : G' \rightarrow H'$ an isomorphism and the vertical maps are a sequence of folds to their respective stiff subgraphs. 
\xymatrix{ G \ar[r]^f \ar[d] & H \ar[d]\\
           G' \ar[r]^{f'}& H'   }
\end{lemma} 
\begin{proof}
The graph map $f': G' \to H' $ is defined as the composition of  the sequence of unfolds from $G'$ to $G$, $f: G \to H$ and sequence of folds from $H$ to $H'$. Then $f':G' \to H'$ is a $\x$-homotopy equivalence between stiff graphs and by \cite[Thm 6.6]{x-htpy} $f'$ is an isomorphism.
 \end{proof}
 
Let $Top$ denote the category of topological spaces with continuous functions.  Recall that if $A$ and $B$ are topological spaces, then a continuous map $i : A \to B$ that has the \hep  \  with respect to every $Z \in Top$ is called a Hurewicz cofibration.

	Dual to the notion of \hep, there is \hlp; \  and a continuous map $p : X \to Y$ in $Top$ with the \hlp \  with respect to every space $Z \in Top$ is called a Hurewicz fibration. Before we proceed further, we give the definition of a model category. 

		\begin{defn}\cite{dwyer}
		A category $\cC$ with three distinguished classes of morphisms - cofibrations, fibrations and weak equivalences is said to be a model category  if it satisfies the following axioms:
		\begin{enumerate}
			\item (Bicompleteness) All the finite limits and finite colimits exists in the category.
			\item (Retract) If $f$ and $g$ are morphisms of the category such that $f$ is a retract of $g$ and $g$	is a weak equivalence, cofibration, or fibration, then so is $f$.
			\item (2 out of 3) If $f$ and $g$ are morphisms of the category such that $gf$ is defined and two
			of $f,\ g$ and $gf$ are weak equivalences, then so is the third.
			\item (Lifting) Acyclic cofibrations\footnote{A morphism that is both a weak equivalence and a cofibration is called an acyclic cofibration.} have	the left lifting property with respect to fibrations, and cofibrations have the left lifting property with respect to acyclic fibrations\footnote{A morphism that is both a weak equivalence and a fibration is called an acyclic fibration.}.
			\item (Factorization) Any morphism of the category can be factored as a cofibration followed by an acyclic fibration, and as acyclic cofibration followed by a fibration.
		\end{enumerate}
		By a model structure on the category we mean a choice of class of cofibrations, fibrations and weak equivalences  satisfying the above properties.
	\end{defn}

	Str\o{}m showed that  $Top$ has a model structure (cf. \cite[Theorem 11]{strom}) with closed Hurewicz cofibrations, Hurewicz fibrations and homotopy equivalences as the class of  cofibrations, fibrations and weak equivalences respectively. We now study analogous constructions in the category of graphs,  $\cG$.
	
	\begin{defn}
Let $A$ and $B$ be graphs and  the maps $j_0: A \to A \times I_n$, $k_0:B\to B\times I_n$ denote graph maps that take an element $x$ to $(x,0)$, for $x$ in $A$ and $B$, respectively. A graph map $i :  A \to B$ is said to have the $\x$-homotopy extension property if for any graph $Z \in \cG$ with graph maps $f : A \to Z, g : B \to Z$, and a $\x$-homotopy $F : A \x I_n \to Z$ satisfying $k_0 i = (i \times 1_{I_n})j_0,   f = F j_0$ and $gi=f$, there exists a $\x$-homotopy $G : B \x I_n \to Z$ that extends $g$, that is, $G (i \times 1_{I_n}) = F$ and $G k_0 = g$.  Equivalently, we require that the dotted arrow should exist for every commutative diagram as in Figure \ref{HEP in G}. 
	
	\begin{figure}[h]
		\centering
		\begin{tikzcd}
			A \arrow[dd, "i"'] \arrow[rr, "j_0"] \arrow[rd, "f"] &  & A \times I_n \arrow[dd, "i \times 1_{I_n}"] \arrow[ld, "F"'] \\
			& Z &  \\
			B \arrow[rr, "k_0"'] \arrow[ru, "g"] &  & B \times I_n \arrow[lu, "G"', dotted]
		\end{tikzcd}
		\caption[x-homotopy extension property in G]{The $\x$-homotopy extension property in $\cG$}
		\label{HEP in G}
	\end{figure}

	Similarly, a graph map $p:X \to Y$  is said to have  the   $\x$-homotopy lifting property if for any graph $A$ and a commutative diagram as in Figure \ref{fig:HLP in G}, there exists a diagonal $G : A \x I_n \to X$ such that $G i_0 = f$ and $pG = g$, where $i_0(a) = (a,0)$.
	
	\begin{figure}[h]
	    \centering
	    \begin{tikzcd}
        A \arrow[r, "f"] \arrow[d, "i_0"']                  & X \arrow[d, "p"] \\
        A \times I_n \arrow[r, "g"] \arrow[ru, "G", dotted] & Y         
        \end{tikzcd}
	    \caption[x-homotopy lifting property in G]{The $\x$-homotopy lifting property in $\cG$}
	    \label{fig:HLP in G}
	\end{figure}
	\end{defn}

By a {\it Str\o{}m-Hurewicz type model structure on $(\cG,\x)$}, we mean a model structure on  $\cG$ where, cofibrations have the $\x$-homotopy extension property, fibrations have the $\x$-homotopy lifting property, and weak equivalences are $\x$-homotopy equivalences.
	
	Let $A , B \in \cG$ be two graphs and  $i : A \to B$ be a graph map. Then the graph $A$ is called a {\it retract} of $B$ if there exists a graph map $r : B \to A$ such that $ri = 1_A$. We call this map $r : B \to A$, a retraction of $B$ onto $A$. We note that if $A$ is a retract of $B$, then $A $ is a subgraph, up to a relabelling of vertices of $A$, of $B$.
	
	Let $G \in \cG$ be a graph, and $f : A \to G$ be a graph map. The \textit{quotient} of $G$ by $f(A)$, denoted by $G / f(A)$ is the graph defined as 
	\noindent $$V(G/f(A))=  V(G)/ V(f(A)), \ 
	E(G/f(A)) = \{ \{[x],[y]\} \subseteq  V(G)/ V(f(A)) \ | \ xy \in E(G)\}.$$

	Let $i:A\to B$ be a graph map. For $n\in\mathbb{N}$, let $(A\x I_n)\bigsqcup\limits_i B$ be the quotient graph such that $V\Big ((A\x I_n )\bigsqcup\limits_i B \Big ) = \Big (V(A\x I_{n})\cup V(B) \Big )/i(a)\sim (a,0)$ and
	$[x][y]\in E\Big ((A\x I_n )\bigsqcup\limits_i B \Big )$ if and only if there exist $x'\in [x], y'\in [y]$ such that
	$x'y'\in E(A\x I_n)\cup E(B)$.\\
	
	Let  $I$ denote the unit interval $[0,1]$ in $Top$. We recall that  a closed continuous map $i : A \to B$ is a closed Hurewicz cofibration if and only if $(A \x I) \cup (B \x \{0\} )$ is a retract of $B \x I$. The analog in graphs is true as well.

	\begin{thm}
		A graph map $i : A \to B$ has the $\x$-homotopy extension property if and only if $(A\x I_n)\bigsqcup\limits_i B$ is a retract of $B \x I_n$, for every $n \in \mathbb{N}$.
	\end{thm}
	
	\begin{proof}
		 Let $i : A \to B$ have the  $\x$-homotopy extension property, that is, for any graph $Z$ and maps $f, F$, and $g$ as in Figure \ref{HEP in G}, there exists a $\x$-homotopy extension $G$. Let $n\in \mathbb{N}$ be fixed, choose $Z = (A\x I_n)\bigsqcup\limits_i B$ together with $F : A \x I_n \to Z$ that sends any $(a,j) \in A \x I_n$ to $(a,j) \in (A \x I_n) \cup B$, $f = F|_{A \x I_0}$ and $g : B \to Z$ as $g(b) = b$ for all $b \in B$. By assumption, there exists $G : B \x I_n \to (A\x I_n)\bigsqcup\limits_i B$ such that $G|_{(A\x I_n)\bigsqcup\limits_i B} = 1_{(A\x I_n)\bigsqcup\limits_i B}$. Therefore, $G$ is a retraction of $B \x I_n$ onto $(A\x I_n)\bigsqcup\limits_i B$.
		
	 Next, let $Z\in\cG$ be any graph with $f:A\to Z$,
		$g:B\to Z$ and $F:A\x I_n\to Z$ be such that $F(a,0)=f(a)$. Let $r:B\x I_n\to (A\x I_n)\bigsqcup\limits_i B$ be a retraction. Define the map  $H:(A\x I_n)\bigsqcup\limits_i B\to Z$ as $H(a,t)=F(a,t)$ for all $a\in A$, $t\in I_n$, and $H(b)=g(b)$    for all $b\in B$. Then $H(a,0)=F(a,0)=f(a)=gi(a)$. Since
		$[(a,0)]=[i(a)]$ in $(A\x I_n)\bigsqcup\limits_i B$, $Hi(a)=gi(a)$.    Since $i(a)\in B$, $H$ is a well defined map. Also, by construction it is a graph map. We define the map $G:B\x I_n\to Z$ as $G=Hr$, then $G$ is a $\x$-homotopy that extends $g:B\to Z$.
	\end{proof}
	
	Let $i : A \to B$ be a graph map. It is easy to see that if $i$ is an isomorphism or  $B$ is  a disjoint union of $A$ and $B - Im(i)$, then $i$ has the $\x$-homotopy extension property. Unlike in $Top$, the converse is also true in $\cG$.
		
	\begin{lemma}
		Let $i : A \to B$ be a graph map such that $i$ is not an isomorphism and $B \ne A \sqcup (B - Im(i))$. Then there does not exist a retraction $r : B \x I_n \to (A\x I_n)\bigsqcup\limits_i B$.
	\end{lemma}
	
	\begin{proof}
		Since  $i$ is not an isomorphism and $B \ne A \sqcup (B - A)$, there exist $a\in A$, $b\in B\setminus i(A)$ such that $i(a)$ is adjacent to $b$ in $B$. Suppose there exists a retraction $r:B\x I_n\to (A\x I_n)\bigsqcup\limits_i B$. Then $(b,0)$ is adjacent to $(i(a),1)$ in $B\x I_n$. Since $r$ is a graph map, $r(b,0)$ is adjacent to $r(i(a),1)$ in
		$(A\x I_n)\bigsqcup\limits_i B$. However, $rj=1_{(A\x I_n)\bigsqcup\limits_i B}$ where $j:(A\x I_n)\bigsqcup\limits_i B\to B\x I_n$ is an inclusion. Therefore $r(b,0)=(b,0)$ and $r(i(a),1)=(i(a),1)$. But $(b,0)$ is not adjacent to $(i(a),1)$ in $(A\x I_n)\bigsqcup\limits_i B$.
		Hence, there does not exist any $r:B\x I_n\to (A\x I_n)\bigsqcup\limits_i B$ such that $rj=1_{(A\x I_n)\bigsqcup\limits_i B}$.    
	\end{proof}

	If there exists a model structure on a category where two of the three special classes of morphisms are chosen to be full morphism class of the category, and the third one is chosen to be isomorphisms in that category, then such a model structure is often called a \textit{trivial} model structure on the category. 
	Droz has shown that:
	
	\begin{thm}[{Droz, \cite[Theorem 4.1]{quillen}}]
		There exists a trivial model structure on the category of graphs, $\cG$, with the class of cofibrations as the isomorphisms of $\cG$ . 
	\end{thm}
	
	Since we are interested in finding the cofibrant replacement of a graph with respect to $\x$-homotopy equivalences as the class of weak equivalences, this structure is not useful for us. Let the class of cofibrations $A\to B$ be isomorphisms or of the form $A \to A \sqcup B'$ and weak equivalences be $\x$-homotopy equivalences. Then regardless of our class of fibrations, it is not possible to factor graph maps $X \to Y$, where both $X$ and $Y$ are connected, as a cofibration followed by an acyclic fibration unless the map itself is a $\x$-homotopy equivalence. Thus we have proved the following theorem.

	\begin{thm}\label{no hurewicz model structure}
There does not exist a Str\o{}m-Hurewicz type model structure on $(\cG, \x)$.
	\end{thm}
	
This theorem is a special case of Theorem \ref{cof as any subclass of induced inclusions then no model structure} which we prove later in the paper. We include  the statement for \Cref{no hurewicz model structure} here since the above argument is  straightforward and does not need the techniques used in the proof of Theorem \ref{cof as any subclass of induced inclusions then no model structure}.

In the next section, we give a characterization of the class of cofibrations given any model structure on $(\cG,\x)$ under mild assumptions.

\section{Criteria for cofibrations in \texorpdfstring{$(\cG,\x)$}{(G,x)}}

 In any model category, the class of fibrations and the class of cofibrations determine each other completely \cite[Proposition 3.13]{dwyer} when we fix the class of weak equivalences. In this section, we analyse the possible classes of maps which would give a compatible choice of cofibrations. In line with our original motivation of replacing graphs with smaller graphs, the  only restriction we impose on the class of cofibrations is that they be inclusions. 

	 We note  that in a model category \cite[Proposition 3.14]{dwyer}, an acyclic cofibration should be preserved under the cobase change along any morphism. Since a $\x$-homotopy equivalence is a composition of folds, unfolds and isomorphisms, we next analyse the behaviour of such maps and inclusions under the cobase change along an arbitrary graph map.
	 
	 	\begin{lemma}\label{cobase changes}
	 	Let $i : A \to B$ be a graph map. If $i$ is an isomorphism, inclusion, induced subgraph inclusion or an unfold then so is the cobase change of $i$ along any graph map.
	 \end{lemma}
	 
	 \begin{proof}
	 	Let $C \in \cG$ be any graph and $f : A \to C$ be a graph map. Let $G \in \cG$ be the pushout of $\{f,i\}$ as shown in Figure \ref{pushout in graphs}.
	 	
	 	\begin{figure}[h]
	 		\centering
	 		\begin{tikzcd}
	 			A \arrow[r, "f"] \arrow[d, "i"] & C \arrow[d, "i'"] \\
	 			B \arrow[r, "f'"'] & G
	 		\end{tikzcd}
	 		\caption{$G$ = Pushout of $\{f,i\}$}
	 		\label{pushout in graphs}
	 	\end{figure}
	 	
	 	It is an easy observation that the cobase change of an isomorphism along any graph map is also an isomorphism. Since set inclusions are preserved under pushouts, the cobase change of a graph inclusion along a graph map will be a graph inclusion. 
		 	
	 	Next, let $i$ be an induced subgraph inclusion, that is, for every $ab \in E(B)$, if $a,b \in V(A)$, then $ab \in E(A)$. As noted earlier $i'$ is an inclusion and therefore it is enough to show that if $[x][y] = i'(x)i'(y) \in E(G)$ for some $x,y \in V(C)$, then $xy \in E(C)$. Again if one of $x$ or $y$ does not belong to  $f(A)$, then $[x][y] \in E(G)$ implies $xy \in E(C)$. So, let $x,y \in f(A)$, that is, for some $a,b \in V(A)$, $f(a) = x$ and $f(b) = y$. Since $[f(a)] = [x]$ is adjacent to $[y] = [f(b)]$, either $i(a)i(b) \in E(B)$ or $xy \in E(C)$. Assume that $i(a)i(b) \in E(B)$ then $i$ being induced subgraph inclusion implies that $ab \in E(A)$. Given that $f$ is a graph map, $ab \in E(A)$ gives $f(a) f(b) = xy \in E(C)$.
	 	
	 	Let $B = A \cup \{b\}$ and there exist $a \in V(A)$ such that $b$ folds  to $a$ in  $B$, that is, $\N_B(b) \subseteq \N_B(a)= \N_A(a)$ (or $\N_A(a)\cup \{b\}$ if $b$ is looped).  Then the pushout graph is $G $ with $V(G) = V(C) \cup \{b\}$ with vertices $x \in V(A)$ identified with $f(x) \in V(C)$. Since an unfold graph map is an induced inclusion, we know that $i'$ is an induced subgraph inclusion. Further, $G$ is a pushout implies that the neighbours of $a$ in $A$ will get identified with neighbours of $f(a)$ in $C$. Therefore,  $\N_G([a])=\N_G([f(a])=\N_C(f(a))$.  Note that $\N_G([b])=f'(\N_B(b))$. Since $\N_B(b)\subseteq \N_A(a)$ (or $N_A(a) \cup\{b\}$ if $b$ is looped) and  $f'$ is a graph map, we get that  $\N_G([b]) =f'(\N_B(b))  \subseteq f'(\N_A(a))$ (or $f'(N_A(a) \cup \{b\}$ if $b$ is looped).  Now $f'(\N_A(a)) =f(\N_A(a)) \subseteq \N_C(f(a))=\N_G([f(a)])$ (or $f'(\N_A(a)) \cup \{b\} \subseteq \N_C(f(a)) \cup \{b\} =\N_G([f(a)]$ if $b$ is looped). Hence $b$ folds to $[f(a)]$ in $G$.	 \end{proof}

		\begin{rmk}\label{rem:induced}
				We first observe that every cofibration must be an induced inclusion. 	Given  a $\x$-homotopy equivalence $i:A\to B$ which is not an induced inclusion, we can always construct a graph map $g : A \to \widetilde{C}_{4}$ (where the graph $\widetilde{C}_{4}$ denotes the 4-cycle having loop at each vertex)  along which the cobase change of $i$ will not be a $\x$-homotopy equivalence. Since $i$ is not induced,
				there exist $a_1,a_2 \in V(A)$ such that $a_1a_2 \notin E(A)$ and $a_1a_2 \in E(B)$. 
		\end{rmk}

	\begin{figure}[H]
	    \centering
				\tikzstyle{black1}=[circle, draw, fill=black!100, inner sep=0pt, minimum width=4pt]
				\tikzstyle{ver}=[]
				\tikzstyle{extra}=[circle, draw, fill=black!00, inner sep=0pt, minimum width=2pt]
				\tikzstyle{edge} = [draw,thick,-]
				\tikzstyle{light} = [draw, gray,-]
				\centering
				
				\begin{tikzpicture}[scale=.36]
				
				\begin{scope}[shift={(-3,0)}]
				\node at (0,7) {$A$};    
				\draw[->] (0,6.5)--(0,0.5);
				\node[left] at (0,3) {$i$};				
				\node[right] at (0,3) {$\sim$};								
				\draw[->] (1,7) -- ( 8.5,7);
				\node at (0,0) {$B$};    		
				\draw[->] (1,0) -- ( 10.5,0);
				\node[above] at (5,7) {$g$};			
				\node[below] at (5,7) {$a_1 \mapsto 1,a_2 \mapsto 3$};	
				\end{scope}
				
				\begin{scope}[shift={(7,5)}]
				\draw[edge] (0,0) rectangle (3,3);	
				\draw [fill] (0,0) circle [radius=0.15];	
				\draw[edge] (0,-0.3) circle [radius=0.3];	
				\node[left] at (-0.1,0.1) {$1$};
				
				\draw [fill] (0,3) circle [radius=0.15];
				\draw[edge] (0,3.3) circle [radius=0.3];	
				\node[left] at (-0.1,3.1) {$2$};			
				
				\draw [fill] (3,0) circle [radius=0.15];
				\draw[edge] (3,-0.3) circle [radius=0.3];				
				\node[right] at (3.1,0.1) {$4$};
				
				\draw [fill] (3,3) circle [radius=0.15];
				\draw[edge] (3,3.3) circle [radius=0.3];	
				\node[right] at (3.1,3.1) {$3$};
				
				\draw[->] (1.5,-1) -- (1.5,-4);		
				\node[right] at (1.5,-2) {$\alpha$};			
			    \draw[->] (3.5,1.5) -- (8.75,-2); 
			    \node at (6.7,0) {$\cong$}; 
				
				\node at (1.5,-5) {$G$};
				\draw[edge] (-1,-4) -- (0,-4) -- (0,-3);
				\draw[->] (2.5,-5) -- (10,-5);
				\node at (6,-4.4) {$\text{folds to}$};
				\end{scope}    
				
				\begin{scope}[shift={(18,-1)}]
				\draw[edge] (0,0) rectangle (3,3);	
				\draw [fill] (0,0) circle [radius=0.15];	
				\draw[edge] (0,-0.3) circle [radius=0.3];	
				\node[left] at (-0.1,0.1) {$x_1$};
				
				\draw [fill] (0,3) circle [radius=0.15];
				\draw[edge] (0,3.3) circle [radius=0.3];	
				\node[left] at (-0.1,3.1) {$x_2$};			
				
				\draw [fill] (3,0) circle [radius=0.15];
				\draw[edge] (3,-0.3) circle [radius=0.3];				
				\node[right] at (3.1,0.1) {$x_4$};
				
				\draw [fill] (3,3) circle [radius=0.15];
				\draw[edge] (3,3.3) circle [radius=0.3];	
				\node[right] at (3.1,3.1) {$x_3$};
				
				\end{scope}
				\end{tikzpicture}
				\caption{}\label{fig:cobase change of non-induced inclusion}
		\end{figure}
		Consider the graph map $g: A \to \widetilde{C}_{4}$ as in Figure \ref{fig:cobase change of non-induced inclusion} that sends $a_1$ to 1, $a_2$ to 3, and rest of $A$ to $2$. Let $G$ be the pushout graph of $\{i,g\}$  and $h$ be the cobase change map of $g$ along $\alpha$.

	Suppose that the cobase change map $\alpha$ of $i$ along $g$ is a $\x$-homotopy equivalence. Since $\widetilde{C}_{4}$ is stiff, $G$ has a subgraph isomorphic to $\widetilde{C}_{4}$, say, $X$ with $V(X) = \{x_1,x_2,x_3,x_4\}$ (see \Cref{fig:cobase change of non-induced inclusion}), to which $G$ folds down. By \Cref{stiff copy}, we get a commuting triangle as in  \Cref{fig:cobase change of non-induced inclusion}, where $\cong$ denotes isomorphism. 
    Now $a_1a_2 \in E(B)$ implies that $h(a_1) h(a_2) \in E(G)$ and therefore, $\alpha(1) \alpha(3) \in E(G)$. For every $1 \le j \le 4$,  $\alpha(j)$ folds in $G$ to some vertex of $X$.  Let $\alpha(1)$ fold to $x_1$. Since $\alpha(j)$  is adjacent to $\alpha(1)$, for each $j = 2,3,4$, it must fold to a vertex in $\{x_1,x_2,x_4\} = V(X) \setminus \{x_3\}$. But this sequence of folds then cannot commute with the isomorphism as $1$ and $3$ are not adjacent in $\widetilde{C}_{4}$ . This contradicts our assumption  that $B$ is $\x$-homotopy equivalent to $A$.
    
   A short and elegant proof of the above argument based on the referee's  suggestion can be obtained by using the notion of Hom complexes and  \cite[Lemma 6.5]{x-htpy}.   	 
\begin{rmk}\label{rem:not all induced inclusions can be cofibs}
	 Next, we observe that in $(\cG , \x)$ not all induced inclusions that are $\x$-homotopy equivalences are preserved under the cobase change along every graph map. 
	
	For instance, let $C$ be the graph on the left bottom side in Figure \ref {cobase change of fold}. Consider the graph $A$ (top left) and $B$ (top right) as the induced subgraphs of $C$ (as shown in Figure \ref{cobase change of fold}) on the vertex sets $ \{a,b,c,p\}$ and $\{a,b,c\}$ respectively. Let $f : A \to B$ be the fold map that sends $p$ to $a$ and $g : A \to C$ be the inclusion. We note that $g$ is a $\x$-homotopy equivalence.

	\begin{figure}[H]
		\tikzstyle{black1}=[circle, draw, fill=black!100, inner sep=0pt, minimum width=4pt]
		\tikzstyle{ver}=[]
		\tikzstyle{extra}=[circle, draw, fill=black!00, inner sep=0pt, minimum width=2pt]
		\tikzstyle{edge} = [draw,thick,-]
		\tikzstyle{light} = [draw, gray,-]
		\centering
		
		\begin{tikzpicture}[scale=.45]
		
		\begin{scope}[shift={(0,0)}]
		\draw[edge] (0,0) -- (3,0) -- (1.5,2) -- (0,0) -- (0,-2);
		\draw [fill] (0,0) circle [radius=0.15];
		\draw [fill] (3,0) circle [radius=0.15];            
		\draw [fill] (1.5,2) circle [radius=0.15];
		\node [right] at    (1.6,1.9)    {$a$};
		\node[left ] at (0,-0.1) {$b$};
		\node[right] at (3,-0.1) {$c$};             
		\draw [fill] (0,-2) circle [radius=0.15];
		\node[left] at (0,-2.1) {$p$};        
		\end{scope}    
		
		\draw[->] (5.5,0)--(11.5,0);
		\node[above] at (8.5,0) {$f$};
		
		\draw[->] (1.6,-3)--(1.6,-6);
		\node[left] at (1.6,-4.5) {$g$};

		\begin{scope}[shift = {(14,0)}]
		\draw[edge] (0,0) -- (3,0) -- (1.5,2) -- (0,0);
		\draw [fill] (0,0) circle [radius=0.15];
		\draw [fill] (3,0) circle [radius=0.15];            
		\draw [fill] (1.5,2) circle [radius=0.15];
		\node [right] at    (1.6,1.9)    {$a$};
		\node[left ] at (0,-0.1) {$b$};
		\node[right] at (3,-0.1) {$c$};             
		\end{scope}    
		
		\draw[->] (15.6,-3)--(15.6,-6);

		
		\draw[->] (5.5,-9)--(8.5,-9);
		\draw[edge] (12,-7)--(13,-7) -- (13,-6);        
		
		\begin{scope}[shift={(0,-9)}]
		
		\draw[edge] (0,0) -- (3,0) -- (1.5,2) -- (0,0) -- (3,-2);
		\draw [fill] (0,0) circle [radius=0.15];
		\draw [fill] (3,0) circle [radius=0.15];            
		\draw [fill] (1.5,2) circle [radius=0.15];
		\node [right] at    (1.6,1.9)    {$a$};
		\node[left ] at (0,-0.1) {$b$};
		\node[right] at (3,-0.1) {$c$};             
		\draw[edge] (0,0) rectangle (3,-2);
		\draw [fill] (0,-2) circle [radius=0.15];
		\draw [fill] (3,-2) circle [radius=0.15];    
		\node[left] at (0,-2.1) {$p$};        
		\node[right] at (3,-2.1) {$q$};
		\end{scope}
		
		\begin{scope}[shift = {(14,-9)}]
		\draw[edge] (0,0) -- (3,0) -- (1.5,2) -- (0,0) -- (1.5,-2) -- (3,0);
		\draw[edge] (1.5,2) -- (1.5,-2);
		\draw [fill] (0,0) circle [radius=0.15];
		\draw [fill] (3,0) circle [radius=0.15];            
		\draw [fill] (1.5,2) circle [radius=0.15];
		\draw [fill] (1.5,-2) circle [radius=0.15];
		\node [right] at    (1.6,1.9)    {$[f(p),f(a),g(a)]$};
		\node[left ] at (0,-0.1) {$[f(b), g(b)]$};
		\node[right] at (3,-0.1) {$[f(c),g(c)]$};        
		\node[below] at (1.5,-2.1) {$[q]$};     
		\end{scope}    
		\end{tikzpicture}
		\caption{Cobase change of an induced inclusion (and of a fold)}
		\label{cobase change of fold}
	\end{figure}
	Then the pushout of $\{f,g\}$ is isomorphic to $K_4$, while $C \simeq_{\x} K_3$. Therefore the cobase change along $f$ of the induced inclusion map $g$ which is also a $\x$-homotopy equivalence  is not a $\x$-homotopy equivalence.  This example indicates that if $g$ is a composition of folds and unfolds then this property  may not be preserved under the cobase change along some graph map even if $g$ is an induced inclusion. \end{rmk}
	
	Based on the above remarks, we know that the class of cofibrations should be a subclass of induced inclusions. Further, it is clear that it has to be a proper subclass of induced inclusions. We now establish that regardless of which subclass of induced inclusions  we choose as the class of cofibrations,  the acyclic cofibrations have to be a subclass of unfolds.

\subsection{ A necessary condition for acyclic cofibrations}

	\begin{defn}\label{relative fold defn}
		Let $A$ be a subgraph of $B$. For some vertex $v \in V(B)$, a fold map $f: B \to B-v $ is called a relative fold in $B$ with respect to $A$, if $A$ is a subgraph of $B-v$.
	\end{defn}
	
	\begin{lemma}\label{cobase change fold \ifif lemma}
		Let $i:A \to B$ be an induced subgraph inclusion and a $\x$-homotopy equivalence. Let $f : B \to B-v$ be a relative fold in $B$ with respect to $A$, $j: A \to B-v$  be the subgraph inclusion and $\alpha: A \to X$ be any graph map. Then the cobase change of the inclusion map $j : A \to B-v$ along $\alpha$
	is a \xhe \ \ifif the cobase change of $i$  along $\alpha$ is.
	\end{lemma}
	
	\begin{proof} 
		Let $g : B-v \to B$ be the unfold map corresponding to the fold map $f : B \to B-v$.
		Suppose the cobase change of  $j : A \to B-v$ along $\alpha$ is a \xhe \ for every graph map $\alpha$ defined on $A$. By Lemma \ref{cobase changes}, an unfold map is preserved under the cobase change along any graph map, therefore any cobase change of $g$ is an unfold map. In particular, any cobase change of $g$ is a \xhe, and hence any cobase change of the composite map $gj = i$ is a \xhe.
		
		Now assume that  the cobase change  of   $i: A \to B$ along any graph map  $\alpha: A\to X$ is a \xhe. We show that  the cobase change of $j : A \to B-v$  along any graph map $\alpha : A \to X$ is also a \xhe. Let $G_1$ be the pushout of $\{j,\alpha\}$, with cobase change maps $j' : X \to G_1$, and $\alpha' : B-v \to G_1$, as shown in Figure \ref{folds preserved \ifif}.
		
		\begin{figure}[H]
			\centering
			\begin{tikzcd}
				A \arrow[r, "\alpha"] \arrow[d, "j"'] & X \arrow[d, "j'"'] \\
				B-v \arrow[d, "g"'] \arrow[r, "\alpha'"] & G_1 \arrow[d, "g'"'] \\
				B \arrow[r, "\alpha''"] & G_2
			\end{tikzcd}
			\caption{}\label{folds preserved \ifif}
		\end{figure}
		
				Let $G_2$ be the pushout of $\{g, \alpha'\}$ along with cobase change maps $g' : G_1 \to G_2$,and $\alpha'' : B \to G_2$. We note that the pushout of $\{ \alpha,gj\}$ is isomorphic to $G_2$, as $j$ and $g$ are inclusions. Then $g'j'$ is a $\x$-homotopy  equivalence by assumption, and $g'$ is a $\x$-homotopy equivalence by Lemma \ref{cobase changes}. By 2 out of 3 property \cite[Lemma 5.5]{x-htpy} of $\x$-homotopy equivalences, this shows that $j'$ is a $\x$-homotopy equivalence. Thus $j$ is also preserved under the cobase change along any graph map $\alpha : A \to X$.
	\end{proof}

	\begin{prop}\label{cobase change of bad folds result}
		Let $i: A \to B$ be an induced subgraph inclusion map and a $\x$-homotopy equivalence  such that $|V(A)| < |V(B)|$. If $B$ does not fold to $A$ via some sequence of relative folds, then there exists a graph map along which the cobase change of $i$ is not a $\x$-homotopy equivalence.
	\end{prop}
	
	\begin{proof}
		Suppose there exists a fold of a vertex $x$ in $B$ relative to $A$. In view of Lemma \ref{cobase change fold \ifif lemma}, the cobase change of $i : A \to B$ along $\alpha$ is a \xhe \  for every map $\alpha$ \ifif the cobase change of $A \to B-x$ is. Therefore without any loss of generality, we assume that for $i: A \to  B$, there is no relative fold in $B$ \wrt $A$. 
		
	In view of \cite[Proposition 6.6]{x-htpy},  given two graphs  $G$ and $H$  with stiff subgraphs $S_G$ and $S_H$ respectively, $S_G \simeq_{\x} S_H$ if and only if $S_G$ is isomorphic to $S_H$.  Since every graph is $\x$-homotopy equivalent to its stiff subgraph, any two $\x$-homotopy equivalent graphs must have isomorphic stiff subgraphs. Therefore $|V(A)| < |V(B)|$ implies that $B$ is not stiff. Thus there exists a $v \in V(B)$ that folds. If $v\in V(B)-V(A)$, then $A \subseteq B-v$ implies  a relative fold in $B$ \wrt  $A$ which contradicts our assumption. Thus any vertex that folds in $B$ is a vertex of $A$. 
		
		Define $A'$ to be the set consisting of vertices of $A$  which fold in $B$.
		For every $x \in A'$, we associate a 5-cycle denoted $C_5^x$ with vertex set $\{x_1,x_2,\dots, x_5\}$ and edge set $\{x_1 x_2, x_2 x_3, x_3 x_4, x_4 x_5,\\ x_5 x_1\}$. We now define a graph $C$ as follows:		$$C = \Big(A \displaystyle\coprod_{x \in A'}C_5^x \Big)/\sim $$
		where `$\sim$' denotes the identification of $x$ with $x_1$, for $x \in A' \subseteq V(A)$.
		 
		 Let $f: A \to C$ be the subgraph inclusion. Consider the pushout object $G$ of $\{i,f\}$ and the cobase change graph  map $i' : C \to G$ of $i$ along $f$.  If the cobase change $i'$ of $i: A \to B$ along $f$ is a $\x$-homotopy equivalence, then the stiff subgraphs of $C$ and $G$ are isomorphic. In particular, cardinality of vertex sets of stiff subgraphs of $C$ and $G$ are equal. We note that $|V(A)| < |V(B)|$, and $A' \subset V(A)$ implies that $|V(C)| < |V(G)|$.
		 
		 If $G$ is not stiff, then there exists a vertex $a \in V(G)$ that folds. We note that a 5-cycle $C_5$ is a stiff graph, and  any of these cycles might have gained a loop on $x_1$'s where it is identified in $G$. However, a 5-cycle with a loop on any of its vertices is also a stiff graph. Therefore, no vertex of $\displaystyle\bigcup_{x \in A'} \{x_2,x_3,x_4,x_5\}$ can fold down to any vertex in $G$. Further,  if $x_1 \in G$  folds  to some vertex in $G$ then $x_2,x_5$ must be a neighbour of that vertex. Since $\N_G(x_2)\cap \N_G(x_5)=\{x_1\}$, this is not possible.
		 
		 If $a\in G$ which folds is not a vertex of these copies of $C_5$ then $a \in V(B)-A'$ and  $\N_G(a)=\N_B(a)$. Since $B$ does not have relative folds with respect to $A$ and the vertices in $B$ which fold are elements of $A'$, there does not exist any such $a\in G$. 		 
		Therefore, the proposition follows. %
	\end{proof}
\section{Lack of model structures on \texorpdfstring{$(\cG,\x)$}{(G,x)}}

	\hspace{-0.75cm}
	\begin{tabular}{m{30em} m{2cm}}
		
		{\quad  We now prove our main theorem that  there does not exist any model structure on $(\cG,\x)$ where cofibrations are a subclass of inclusions. Define  $L_n$ to be the graph obtained from the path graph $P_n$ by adding a loop at a degree 1 vertex, say $0$, that is, $V(L_n) = \{0,1,\dots,n\}$, and $E(L_n) = \{xy: |x-y| = 1\} \cup \{00\}$. Let $p_n: L_n \to I_0$ be the graph map that sends each vertex of $L_n$ to $0$, as shown on the right. }
		
		&{  
			\begin{tikzpicture}[font=\tiny][baseline=(current bounding box.north), level/.style={sibling distance=50mm}]
			\tikzstyle{edge} = [draw,thick,-] 		
			\draw[edge] (0,2) -- (0,1);
			\draw[dotted,thick] (0,1) -- (0,0.5);
			\draw[edge] (0,0.5) -- (0,0);
			\draw[edge] (0,2.2) circle [radius=0.2];
			\draw [fill] (0,2) circle [radius=0.1];			
			\draw [fill] (0,1.5) circle [radius=0.1];						
			\draw [fill] (0,1) circle [radius=0.1];		
			\draw [fill] (0,0.5) circle [radius=0.1];									
			\draw [fill] (0,0) circle [radius=0.1];				
			
			\node[right] at (0.1,2) {$0$};
			\node[right] at (0,1.5) {$1$};
			\node[right] at (0,1) {$2$};
			\node[right] at (0,0.5) {$n-1$};			
			\node[right] at (0,0) {$n$};
			
			\draw[edge, ->] (0,-0.3) -- (0,-0.9);
			
			\draw[edge] (0,-1.3) circle [radius=0.2];
			\draw [fill] (0,-1.5) circle [radius=0.1];	
			\node[right] at (0.1,-1.5) {$0$};								
			\node[right] at (0,-0.6) {$p_n$};											
			\end{tikzpicture}
}
	\end{tabular}

	 \begin{lemma} \label{unfold has llp wrt L_n to loop}
	Every unfold map has the \llp \ \wrt \ $p_n$ for every $n$.
	\end{lemma} 
	\begin{proof}
		Let $i: X \to X \cup x$ be an unfold map, and $f : X \to L_n$, $g : X \cup x \to I_0$ be any graph maps such that $p_n f = gi$ (see \Cref{LLP wrt L_n to i_0}).
		 Since image of $p_n$ is a single looped vertex, $g(x) = g(x')$, where $x'$ is a vertex from which $x$ is unfolded. Then the map $F : X\cup x \to L_n$ defined as $F|_{X} = f$ and $F(x) = f(x')$, is a graph map that satisfies $Fi = f$ and $p_n F = g$. This implies that every unfold map has \llp \ \wrt \ $p_n$.

	\begin{figure}[H]
		\centering
		\begin{tikzcd}
			X \arrow[r, "f"] \arrow[d, "i"'] & L_n \arrow[d, "p_n"] \\
			X \cup x \arrow[r, "g"] \arrow[ru, "F", dotted] & I_0
		\end{tikzcd}
		\caption{The map $p_n: L_n \to I_0$ has the right lifting property with respect to any unfold} \label{LLP wrt L_n to i_0}
	\end{figure}
	\end{proof}

	\begin{cor}\label{cor:p_n is acyclic fibration}
		If $\mathcal{M}$ is a model structure on $\cG$ whose weak equivalences are $\x$-homotopy equivalences, and every acyclic cofibration is a  composition of unfolds,  then the graph map $p_n$ defined above is an acyclic fibration.
	\end{cor}

	\begin{proof}
      	Note that $L_n$ folds to $I_0$, and hence for any $n$, $p_n$ is a $\x$-homotopy equivalence.
		By  Lemma \ref{unfold has llp wrt L_n to loop}, $p_n$ lifts on the right of every unfold. 
		Therefore, for every $n \in \mathbb{N}$, $p_n$ will lift against all compositions of unfolds and hence against acyclic cofibrations. Since in a model structure any map which lifts against all acyclic cofibrations is a fibration \cite[Prop 3.13]{dwyer} we get that $p_n$ is an acyclic fibration.  
	\end{proof}
	
	\begin{defn}

		Distance between any two vertices $u,v$ of a graph $G$ is the number of edges in a shortest length path connecting the two vertices, denoted $dist_G(u,v)$. For two subgraphs $G_1,G_2$ of $G$, the {\it distance} between them, $dist_G(G_1,G_2)$, is defined 
		as $$dist_G(G_1,G_2) = \min \{dist_G(u,v): u \in V(G_1), v \in V(G_2)\}.$$
	\end{defn}

	\begin{lemma}\label{lem:pn not lifting}
    	Let $i:K_2 \to C_n$  be an inclusion and $n$ odd. Then $i$ does not lift against $p_{t}$ for any natural number $t>n$.
	\end{lemma}
	\begin{proof}
	   Define $g : C_n \to I_0 $ to be the constant graph map.  Define $f : K_2 \to L_t$ as $f(1) = t, f(2) = t-1$. Since both  $p_t$ and $g$ are constant maps, $p_t f = gi$. However, $t > n$ implies that given any graph map $F : C_n \to L_t$ such that  $p_t F =g$ and $F i =f$, image of $F$ does not include the vertex $0$. Therefore, $Im(F)$ is a simple subgraph of $L_t$ and in particular, is a path graph say $P_l$. Thus existence of $F$ implies that there is a map from an odd cycle $C_n$ onto a path graph. This is not possible, since a graph map $C_{n} \to P_l$ implies that the chromatic number of $C_n$ cannot be bigger than the chromatic number of $P_l$. 
	\end{proof}
	More generally, the above argument can be used to prove that for any graph $B$, if $f: K_2 \to B$ factors through an odd cycle then it does not lift against $p_k$ for some $k$.
	
	\begin{thm}\label{cof as any subclass of induced inclusions then no model structure}
		If the class of cofibrations is any subclass of induced subgraph inclusions, then there does not exist a model category structure on $(\cG,\x)$.
	\end{thm}
	\begin{proof}
	Suppose  $(\cG,\x)$ has a model category structure with the class of cofibrations $\cC$ as a subclass of all induced subgraph inclusions. By Proposition \ref{cobase change of bad folds result}, the class of acyclic cofibrations has to be a subclass of all compositions of unfolds.
	Further by Corollary \ref{cor:p_n is acyclic fibration}, we know that in any such model structure, the graph map $p_n: L_n \to I_0$ must be an acyclic fibration.

	Consider $f : K_2 \to  C_n$, for $n$ odd. 
    Clearly, this is not a $\x$-homotopy equivalence. Further, $f$ is an inclusion  by \Cref{lem:pn not lifting}, we see that $f$ is not a cofibration. 
    Then $f$ can be factored into a cofibration $K_2 \to B$ and an acyclic fibration $B \to C_n$, with neither of these maps being isomorphisms.
            
    Since any cycle graph $C_n$ for $n \ne 4$ is a stiff graph, $B$ has an isomorphic copy of $C_n$, say $\widetilde{C_n}$, to which the image of $K_2$, say $\widetilde{K_2}$, in $B$ must fold. If $\widetilde{K_2}$ is in $\widetilde{C_n}$, then  \Cref{lem:pn not lifting} implies $f$ cannot be a cofibration. If  $\widetilde{K_2}$ is not in $\widetilde{C_n}$, then  the edge $v_1v_2$ of $\widetilde{K_2}$ will fold down via a sequence of folds to an edge in $\widetilde{C_n}$.  Let the distance  from $v_1v_2$ to the edge it folds down to in $\widetilde{C_n}$ be $r$. Then an  argument similar to one used  in \Cref{lem:pn not lifting} will show that  $K_2 \to B$ does not lift against $p_{r+2+t }$ for all natural numbers $t>n$. This implies $K_2 \to C_n$ cannot be factorized and  
		hence there is no such model structure on $(\cG,\x)$ with $\cC$ as the cofibrations.
	\end{proof}

  As elucidated earlier, the non-existence of a model structure on $(\cG,\x)$ where the class of cofibrations are a subclass of inclusions arises from the lack of compatibility with any class of fibrations. Going forward, it is reasonable to look for a structure which allows us to define homotopy pushouts with a less rigid structure. One such structure which can facilitate the existence of homotopy pushouts is that of a cofibration category (for the definition, refer to \cite{cofibration-category-defn}). In our forthcoming article we explore the existence of a  cofibration category structure on $(\cG,\x)$.

\subsection*{Acknowledgements} We thank the anonymous referee for several suggestions which has improved this article.


\end{document}